\documentclass[12 pt]{article}
\usepackage[utf8]{inputenc}
\usepackage{authblk}
\usepackage{setspace}
\usepackage[margin=1.25in]{geometry}
\usepackage{graphicx}
\graphicspath{ {./figures/} }
\usepackage{subcaption}
\usepackage{amsmath}
\usepackage{lineno}
\usepackage[english]{babel}
\usepackage{amsthm}
\newtheorem{theorem}{Theorem}[section]
\newtheorem{lemma}[theorem]{Lemma}
\newtheorem{corollary}{Corollary}[theorem]

\title{Asymptotic Lower Bounds for the Feedback Arc Set Problem in Random Graphs}

\author[1*]{Harvey Diamond}
\author[2\dag ]{Mark Kon\thanks{$\dag$Research partially supported by National Science Foundation grant DMS-23-19011.}}
\author[3]{Louise Raphael}

\affil[1] {\footnotesize{Department of Mathematics, West Virginia University, Morgantown, WV 26506}}
\affil[2]{Department of Mathematics and Statistics, Boston University, Boston, MA, 02215}
\affil[3] {Department of Mathematics, Howard University, Washington, DC 20059}
\affil[*]{Address correspondence to: harvey.diamond@mail.wvu.edu}

\date{}

\begin{document}
\maketitle

\begin{abstract}
\begin{linenumbers}
Given a directed graph, the Minimum Feedback Arc Set (FAS) problem asks for a minimum (size) set of arcs in a directed graph, which, when removed, results in an acyclic graph.
In a seminal paper, Berger and Shor [1], in 1990, developed initial upper bounds for the FAS problem in general directed graphs.  Here we find asymptotic \textit{lower bounds} for the FAS problem in a class of random, oriented, directed graphs derived from the Erd\H{o}s-R\'{e}nyi model $G(n,M)$, with n vertices and M (undirected) edges, the latter randomly chosen. Each edge is then randomly given a direction to form our directed graph.  We show that

$$Pr\left(\textbf{Y}^* \le M \left( \frac{1}{2} -\sqrt{\frac{\log n}{\Delta_{av}}}\right)\right)$$ approaches zero exponentially in $n$, with $\textbf{Y}^*$ the (random) size of the minimum feedback arc set and $\Delta_{av}=2M/n$ the average vertex degree. Lower bounds for random tournaments, a special case, were obtained by Spencer [12] and de la Vega [13] and these are discussed.  
In comparing the bound above to averaged experimental FAS data on related random graphs developed by K. Hanauer [7] we find that the approximation $\textbf{Y}^*_{av} \approx M\left( \frac{1}{2} -\frac{1}{2}\sqrt{\frac{\log n}{\Delta_{av}}}\right)$ lies remarkably close graphically to the algorithmically computed average size $\textbf{Y}^*_{av}$ of minimum feedback arc sets. \par 
\noindent
\textit{Keywords}  feedback arc set, random graphs, asymptotic lower bounds
 \end{linenumbers}
\end{abstract}


\section{Introduction}

In a directed graph (digraph) $D =(V,E)$ with $V$ the set of vertices and $E$ the set of directed arcs, the Minimum Feedback Arc Set (FAS) problem is to find a minimum (size) subset $E'$ of arcs such that removing these arcs produces a digraph $D'=(V,E \setminus E')$ that is acyclic. As noted in the classic text of F. Harary [8], a digraph is acyclic if and only if it is possible to order the vertices of $D$ so that every directed edge goes from a lower-numbered to a higher-numbered vertex. As such, in the FAS problem, we are renumbering the vertices so that when the feedback arcs (from higher-numbered to lower-number vertices) are removed the set of remaining (feedforward) arcs is as large as possible. Following notation introduced in Eades [4] we let $F(D)$ denote a feedback arc set, so that $D(V,E\setminus F(D))$ is acyclic. $F^*(D)$ denotes a feedback arc set of minimum cardinality $|F^*(D)|$. We use $n$ to denote $|V|$ and $M$ to denote $|E|$. An  \textit{oriented directed graph} is one which has no two-cycles (arcs in both directions between two vertices). In what follows, the \textit{degree} of a vertex is the sum of its in-degree and out-degree.

The FAS problem is longstanding in the field of graph theory and its applications. While originating as a problem in circuit theory in a paper by D. Younger [14], algorithmic interest as a graph theoretic problem can be dated from its identification as an NP-complete problem in Richard Karp's landmark 1972 paper [10] containing a compilation, with proofs, of important NP-complete combinatorial problems. We note in this regard the recent survey text by Robert Kudelić [11] of results over the subsequent 50 years, primarily concerning the development of practical algorithms for this problem. In the introductory chapter, Kudelić ([11], page 11) provides a listing, with references, of over 30 applications in computing, operations research, and other application areas. 

\begin{linenumbers}
    
As an NP-complete problem, a minimum feedback arc set is difficult to find, so that it is useful not only to develop effective algorithms, but to determine upper and lower bounds for the minimum feedback arc set, the setting within which this paper lies.  For general directed graphs, the seminal works by Bonnie Berger and Peter W. Shor [1] in 1990 and following, develop upper bounds, with accompanying algorithms. In [1], they obtain the result 
\end{linenumbers}
$$|F^*(D)| \le M(1/2-\Omega(\Delta^{-1/2}))$$
and further exhibit a special class of directed graphs for which 

$$|F^*(D)| \ge M(1/2-O(\Delta^{-1/2}))$$

\noindent where $\Delta$ is the maximum vertex degree, showing that their general upper bound is sharp with respect to order. Their results and algorithms have since been extended and sharpened in many ways, e.g. the recent paper by Jacob Fox, Zoe Himwich, and Nitya Mani [6]. 

Significant early results on bounds for the FAS problem in tournaments were obtained by Erd\H{o}s and Moon [5], followed by Spencer [12], and subsequently improved by de la Vega  [13] using a different proof technique. 

This paper considers the problem of \textit{lower bounds} for $|F^*(D)|$ for an important general family of directed, oriented random graphs, associated with the well-known Erd\H{o}s-R\'{e}nyi $G(n,M)$ model, as described in Bollob\'{a}s [2] . In the model considered here, $M$ undirected arcs between $n$ vertices are chosen at random as in the $G(n,M)$ construction, and then the directed graph $\textbf{D}$ is obtained by randomly choosing, for each arc, a direction with equal probability. We apply boldface to the directed graph $\textbf{D}$
 if it is obtained via a random construction. 

Lower bounds seem to have not received much attention in the literature, perhaps because for many classes of graphs, the lower bound for $|F^*(D)|$ is $0$.  While an upper bound tells you to keep working if the size of your feedback set exceeds it, a lower bound tells you when it may not be effective to keep working, if you have already obtained a set not much larger in size than that bound. We would like to get a sense of how likely it is that minimum feedback arc sets of a given size occur, within the setting of the $G(n,M)$ random model. In fact, as a special case we show that if $\frac{n\log(n)}{M}\rightarrow 0$, then for any given $\epsilon>0$ the probability that $|F^*(\textbf{D})|>(M/2)(1-\epsilon)$ approaches one as $n\rightarrow \infty$. An asymptotic bound for $\epsilon=\epsilon(n,M)$ that is sufficient to maintain that probabilistic result is our main result, obtained in Section 3. In this connection we note the following result (slightly paraphrased) which appears at the end of Erd\H{o}s and Moon [5] which applies to oriented directed graphs: 

\begin{quote} 

(Erd\H{o}s and Moon) If $f(n,M)$ is the greatest (integer) lower bound for the maximum size of an acyclic subgraph in every oriented directed graph with $n$ vertices and $M$ edges, and we assume $n\log(n)/M \rightarrow 0$ as $M,n \rightarrow \infty$ then $\displaystyle {\lim_{n \rightarrow \infty} \frac{f(n,M)}{M}=1/2}$.

\end{quote}

\noindent The import of this result in the present case is that for large enough values of $M$, $n$ satisfying the asymptotic condition given, for any given $\epsilon>0$ there are oriented directed graphs for which $|F^*(D)|>(M/2)(1-\epsilon)$. This only guarantees the existence of such graphs but doesn't tell us how numerous they are. But in fact, as noted above, almost all such large graphs are ``bad'' in this sense of requiring almost half the arcs to be removed in order to obtain an acyclic remaining subgraph. We would note that the present authors derived this property for the $G(n,p)$ random model  in their earlier paper [3]. In this model (see Bollob\'{a}s [2] ), where two-cycles are possible, each possible arc from vertex i to vertex j is chosen independently with probability p. Our current $G(n,M)$ model is actually easier to analyze mathematically, and it naturally results in oriented, directed graphs, which are more commonly the subject of research in this area.

Of particular salience to our development is the result of de la Vega [13] for random tournaments, which can be restated for minimum feedback arc sets as follows:
$$Pr\left(|F^*(\textbf{D})| \ge \frac{M}{2}-1.73n^{3/2}\right) \rightarrow 1 \text{ as } n\rightarrow \infty .$$

These random tournaments are special cases of our random graphs where $M=\binom{n}{2}$, i.e. every arc is included but the direction is random. A similar bound applicable to our more general random graphs would therefore be desirable but is seemingly not available with the techniques used here. 

The remainder of the paper is organized as follows: 

In Section 2 we formulate our problem mathematically and develop some basic probabilistic inequalities in considering the $n!$ vertex orderings. Separately, we obtain inequalilties that apply to the tails of the binomial distribution $B(M,1/2)$. In Section 3 we use these inequalities to obtain our lower bound for the size of the minimum feedback arc set. The provided lower bound is a probabilistic one, where the probability of a smaller $|F^*(\textbf{D})|$ than our bound is exponentially small in $n$. The combinatorial and probabilistic techniques employed are not particularly novel, but we believe our results are new and provide another useful ``datapoint'' in understanding the size of minimum feedback arc sets in widely used classes of directed graphs. 

In Section 4, we compare our lower bound with a robust set of experimental results from FAS algorithms applied to a suite of random graphs, as obtained by Kathrin Hanauer in her Master's thesis [7]. We discover that a small modification of our lower bound produces a direct formula approximating the minimum FAS size that graphically matches the experimental data with striking accuracy over a range of edge densities and number of vertices. 

\section{Formulation and  Basic Inequalities}
In conceptualizing this problem we find it convenient to work with the adjacency matrix $A(D)$ of a directed graph $D$. This is a $0/1$ matrix where a $1$ in the $(i,j)$ position corresponds to an arc  from vertex $i$ to vertex $j$. Hence a $1$ where $j>i$ above the diagonal indicates an arc from a vertex $i$ to a higher-numbered vertex $j$. A $1$ below the diagonal corresponds to an arc from a higher-numbered vertex to a lower numbered vertex. The latter are called feedback arcs.  The FAS problem is then to find a renumbering of the vertices that results in a minimum number of $1's$ below the diagonal. These $1's$ identify the feedback arcs to be eliminated. 

As discussed in the papers [1]  and [6], in which various bounds and algorithms are considered, we will use $\textit{oriented directed graphs}$, which are undirected graphs for which each edge is assigned a direction. Equivalently, these are directed graphs without 2-cycles.  Our goal in this paper is to obtain $\textit{lower}$ bounds on the size of the minimum  feedback arc set, that apply with high probability, asymptotically increasing to 1 with the number of vertices, in a class of oriented directed, random graphs. 

The graphs we consider will be random oriented directed graphs on $n$ vertices generated from the Erd\H{o}s-R\'{e}nyi $G(n,M)$ model, in which we have $n$ vertices and then $M$ randomly chosen undirected arcs between the vertices. Each such arc is then randomly (with probability 1/2) assigned a direction, resulting in an oriented directed graph $\textbf{D}$. Considered as a random experiment, starting with a zero $n \times n$ matrix we are choosing $M$ locations above the diagonal in which to insert a $1$, and then each such $1$ located in the $(i,j)$ position, $j>i$, is relocated to the $(j,i)$ position below the diagonal with probability $1/2$. We note that in a previous arXiv paper  [3] the authors considered the model referred to there as $D(n,p)$ where directed graphs are constructed so that each possible directed arc is included in the graph with probability p. Probability bounds similar to those of Section 2 were obtained there. However, the model $G(n,M)$ used here is more germane to research in the field, and we also have available, as discussed in Section 4, a collection of related experimental results on the size of the feedback arc set that can be readily compared with the theoretical results presented in Section 3. We would remark finally that the random graphs considered here, with their statistically uniform structure, are not typical of practical examples, but they can serve nevertheless as a useful reference example and are worth considering for their mathematical interest. 

We begin with some notation and terminology. In our random graph construction, we define the random outcome as the adjacency matrix $\textbf{A}$ with $\textbf{D(A)}$ as the associated directed graph, and the random variables  $\textbf{X(A)}$ and $\textbf{Y(A)}$ as the number of 1's above and below the diagonal, respectively. We observe that $\textbf{Y}$ is a binomial random variable with binomial distribution $B(M,1/2)$ and $\textbf{X}+\textbf{Y}=M$. We refer to a realizable arrangement of 0's and 1's under this construction as a $\textit{configuration}$. Boldface generally indicates the outcome of a random experiment, while normal typeface e.g. $A,X,Y$
refer to particular or generic nonrandom cases. By a permutation of an adjacency matrix $A$ we mean a matrix $P^TAP$ for some permutation matrix $P$, or equivalently, the adjacency matrix of some vertex reordering of the directed graph equivalent to $A$. 

For any given directed graph considered here (whether random or not), with adjacency matrix $A$ we define $A^*(A)$ as an adjacency matrix that results from a solution of the FAS problem, with corresponding numbers $X^*$ and $Y^*$ of ones above and below the diagonal respectively with $X^*+Y^*=M$.  Thus $A^*=P^TAP$ for some $n \times n$ permutation matrix $P$ and and $Y^*$ is the minimum number of ones that appear below the diagonal among all such permutations. Correspondingly, given an adjacency matrix $A^*$ with a minimum set of feedback arcs, the adjacency matrices $\{P^TA^*P\}$ are precisely the adjacency matrices that have $A^*$ as an optimal reconfigured feedback arc set.  The matrix entries of $A^*$ we refer to as an $\textit{optimal configuration}$

We now develop the basic inequalities used in the next section to obtain our lower bound. We begin with a simple lemma providing a bound on the probability that the minimum feedback arc set has no more than $k$ edges, i.e., the event $\textbf{Y}^* \le k$. Because there are $n!$ permutations of the vertex ordering, a minimum feedback arc set with no more than $k$ arcs can have originated from any of the $n!$ permutations of its corresponding adjacency matrix. If we enlarge this set of adjacency matrices to those with $Y\le k$, the set of $n!$ permutations of this set cannot be any smaller. This gives rise to a simple inequality:

\begin{lemma} For our random graph construction, we have the following:

$Pr(\textbf{Y}^* \le k) \le n!Pr(\textbf{Y} \le k)$.

\end{lemma}

\begin{proof}
First we show that the sets $E^*=\{A:Y^*(A) \le k\}$ and $E=\{A: \exists P, Y(P^TAP) \le k\}$ are the same. Clearly if $A \in E^*$ then for some permutation matrix $P$, we have $Y(P^TAP)=Y^* \le k$, showing that  $E^*\subset E$. On the other hand, every adjacency matrix $A$ in $E$ has a reordering with at most $k$ feedback arcs, so that $Y^*(A) \le k$, showing that $E \subset E^*$. Now the set $E$ consists of all the permutations of matrices $A$ for which $Y(A) \le k$ since every permuation is invertible via a permutation. Hence $E$ consists of the union, over the $n!$ different permutations $P$,  of the set $\{A:Y(A) \le k\} $, namely the set $E=\bigcup_{P}\{P^TAP:Y(A) \le k\}$. Each such set in the union has the same probability and by the subadditivity property of probability of a union, we have $Pr(\textbf{Y}^* \le k) \le n!Pr(\textbf{Y} \le k)$.
\end{proof}
We note that our sets are not in general disjoint since $Y^*(A)<Y(A)$ will be true in some cases, and the same matrix $P^TAP$ may arise from two different $A,P$ for which 
$Y(A) \le k$ . The utility of the result is that $Pr(\textbf{Y}(\textbf{A}) \le k)$ can be readily estimated and provides the inequality we are looking for. 

Under the random construction, the random variable $\textbf{Y}$ has the binomial distribution $B(M,1/2)$. We are interested in the probability that $\textbf{Y} \le k$ for some $k<M/2$, as our question is how small can we expect the minimum feedback arc set to be under this random construction.

\begin{linenumbers}
\begin{theorem} 
 $Pr(\textbf{Y} \le M ( \frac{1}{2} -t)) \le \exp(-2Mt^2)$.
\end{theorem}

\begin{proof}

This result follows from the well-known theorem of Wassily Hoeffding [9]:
     Let $\textbf{X}_1,..,\textbf{X}_n$ be independent random variables satisfying $a_i \le \textbf{X}_i \le b_i$ almost surely. Set $\textbf{S}_n=\textbf{X}_1+...+\textbf{X}_n$. Then
    $$Pr(\textbf{S}_n-E[\textbf{S}_n] \ge t) \le \exp \left(-\frac{2t^2}{\sum_{i=1}^{n} (b_i - a_i)^2}\right).$$
The theorem applies as well to the left tail, namely to $Pr(\textbf{S}_n-E[\textbf{S}_n] \le -t)$. 
Applying Hoeffding's theorem in this form to our case, where $\textbf{Y}=\sum_{i=1}^{M}{\textbf{X}_i}$, with $0 \le \textbf{X}_i \le 1$ and $E[\textbf{Y}]=M/2$ we obtain 
$$Pr(\textbf{Y}-M/2 \le -t) \le \exp \left(-\frac{2t^2}{M}\right) \text{,}$$
and then replacing $t$ by $Mt$, the result follows. 

\end{proof}

\end{linenumbers}
Remark: The variance of $\textbf{Y} \sim B(M,p)$, with $p=1/2$ is $M/4$. The normal approximation of the binomial distribution would suggest that

$$Pr(\textbf{Y}-M/2 \le -s \sqrt{M/4}) \sim \frac{1}{\sqrt{2\pi}}\int_{-\infty}^{-s} \exp(-u^2/2)du$$ 

\noindent while Theorem 2.2 provides the inequality (using $t=s/\sqrt{4M}$ )

$$Pr(\textbf{Y}-M/2 \le -s \sqrt{M/4}) \le \exp(-s^2/2)$$

\noindent so the Hoeffding bound gives us the "correct exponent" even for small values of $s$ and the expressions differ only by an algebraic factor $O(1/s)$ that will not significantly affect the results were we to seek a more sophisticated inequality for the tails.

\section{Asymptotic Lower Bounds}
As noted in the introduction, Berger and Shor [1] developed a randomized algorithm demonstrating an $\textit{upper}$ bound on the minimum feedback arc set of at most $M(\frac{1}{2}-\frac{C}{\sqrt{\Delta}})$, where $M$ is the number of edges, $\Delta$ is the maximum vertex degree and $C>0$ is a constant that depends only on $\Delta$. Here we develop for our class of random graphs a probabilistic $\textit{lower}$ bound of a related form, with probability converging exponentially to one, asymptotically as the number of vertices $n$ increases to infinity. 

From Lemma 2.1 we have $Pr(\textbf{Y}^* \le k) \le n!Pr(\textbf{Y} \le k)$ and Theorem 2.2 provides the inequality $Pr(\textbf{Y} \le M ( \frac{1}{2} -t)) \le \exp(-2Mt^2)$ from which it follows that 
\begin{equation}
    Pr \left(\textbf{Y}^* \le M ( \frac{1}{2} -t) \right) \le n!\exp(-2Mt^2).
\end{equation}
The interplay between $n$,$M$, and $t$ provides an opportunity to develop conditions for which $Pr(\textbf{Y}^* \le M ( \frac{1}{2} -t))) \rightarrow 0$. One such set of conditions gives us the following theorem:
\begin{theorem} If $\Delta_{av}=\frac{2M}{n}$ denotes the average vertex degree, and $\textbf{Y}^*$ is the size of the minimum feedback arc set in random graphs of the $G(n,M)$ model, we have
    $$Pr\left(\textbf{Y}^* \ge M \left( \frac{1}{2} -\sqrt{\frac{\log n}{\Delta_{av}}}\right)\right) \ge 1-3\sqrt{n}e^{-n} \text{.}$$
\end{theorem}
\begin{proof}

\noindent Stirling's formula provides that $\log (n!) \le n \log n -n+\frac{1}{2} \log (2\pi n)+\frac{1}{12n} $. In Equation (1) we express $n!=\exp(\log(n!))$, and then using Stirling's formula, we choose the value of $t$ that will cancel the $n\log n $ term in the exponent, namely $t=\sqrt{\dfrac{\log n}{\Delta_{av}}}$. We then have

    $$Pr \left( \textbf{Y}^* \le M \left( \frac{1}{2} -\sqrt{\frac{\log n}{\Delta_{av}}}\right) \right) \le \sqrt{2\pi n}\exp\left(-n+\frac{1}{12n}\right)\le 3\sqrt{n}e^{-n}\text{,}$$
    rounding up to a bound $3$ for the constant factor. Finally, we reverse the inequality, obtaining the Theorem in terms of a lower bound that holds asymptotically with probability converging to $1$ exponentially, as $n \rightarrow \infty$.  

\end{proof}

We observe that the lower bound in Theorem 3.1 is clearly only useful when $M$ increases with $n$ in such a way that $\sqrt{\dfrac{\log n}{\Delta_{av}}}=\sqrt{\dfrac{n\log n}{2M}}$ is less than $1/2$, and in most practical cases we would like $\log n=o(\Delta_{av})$. In particular, we can observe the following corollary:
\begin{corollary}
    For any $\epsilon>0$, if $\frac{n\log n}{M}\rightarrow 0$ then $Pr\left(\textbf{Y}^* \ge M ( \frac{1}{2} -\epsilon)\right)\rightarrow 1$. 
\end{corollary}
Thus, in this random setting, the same condition $\frac{n\log n}{M}\rightarrow 0$ noted in 
Erd\H{o}s and Moon [5] guarantees that $|F^*(\textbf{D})| \ge  M ( \frac{1}{2} -\epsilon)$ with probability approaching 1 as $n \rightarrow \infty$. Almost ``all'' such graphs with $n$ vertices and $M$ arcs are ``bad'' in requiring that nearly $M/2$ arcs must be removed to obtain an acyclic subgraph. 
\section{Application to Experimental Results}
In her Masters thesis  [7], Kathrin Hanauer investigated the experimental performance of eight FAS algorithms on random graphs, applying further post-processing algorithms in some cases to reduce the size of the set obtained. Among the examples considered were graphs generated from the random construction referred to in the thesis as $ER(n,p)$, in which the undirected Erd\H{o}s-R\'{e}nyi model $G(n,p)$ is the starting point, with each possible arc included with probability $p$, after which a direction is assigned with equal probability.  We compare these results with the lower bounds predicted here for the model $G(n,M)$, choosing for $M$ the expected number of edges, namely $M=p\binom{n}{2}=p\frac{n(n-1)}{2}$. Hanauer ran 20 examples for each $(n,p)$ pair and plotted the average size minimum feedback arc set calculated by each algorithm, so the lower envelope of this data is reasonably taken as an approximation of the size of actual average minimum feedback arc set for random graphs $ER(n,p)$. This plot will be compared with our asymptotic lower bound as calculated from the expected number of edges.  We work directly off the figures that appear in the thesis, as the original data was no longer available. 
\begin{linenumbers}
Using Hanauer's figures from her thesis, we plot the curves
$\Bigl(n, M \bigl( \frac{1}{2} -\sqrt{\frac{\log n}{\Delta_{av}}}\bigr)\Bigr)$
for the given value of $p$, where $M=p\binom{n}{2}$ and $\Delta_{av}=2M/n$. Somewhat by chance, we observed that by halving the deviation from $M/2$ given by $M \sqrt{\frac{\log n}{\Delta_{av}}}$ we create a direct heuristic approximation for the expected minimum size feedback arc set that is nearly indistinguishable from Hanauer's experimental average minimum size FAS curves in most cases, and we include this heuristic curve as well, representing the approximation $\textbf{Y}^*_{av}\approx M \left( \dfrac{1}{2} -\dfrac{1}{2} \sqrt{\dfrac{\log n}{\Delta_{av}}}\right)$ with $\textbf{Y}^*_{av}$ the algorithmically computed experimental average of $\textbf{Y}^*$ . The theoretical lower bound is plotted in orange and the  heuristic approximation is plotted in red. We plot directly on the figures from [8], which include the original captioning. To carry out the plot, the axes were digitized by finding screen coordinates using convenient demarcated points on the axes in the images.  
\end{linenumbers}
\begin{linenumbers}

In Figure 1, plotting from $n=20$, the left figure involves the random graphs $ER(n,p)$ for $p=1/2$, which we will compare to our case where we use $M=\frac{1}{2}\frac{n(n-1)}{2}$. The two curves plotted are therefore

$$\left( n, M \left( \dfrac{1}{2} -\sqrt{\dfrac{\log n}{\Delta_{av}}}\right)\right) \quad \text{and} \quad  
\left(n, M \left( \dfrac{1}{2} -\dfrac{1}{2} \sqrt{\dfrac{\log n}{\Delta_{av}}}\right)\right),$$ 
reflecting our lower bound and our heuristic approximation, respectively. For these experimental results, the thesis notes there was additional post-algorithmic optimization carried out, and hence the experimental estimates of $|F^*(D)|$ even more reliable. On the right in Figure 1 is the case of $p=1$ meaning that we start with every possible (undirected) arc on $n$ vertices. This is in fact a random tournament, so the bound of de la Vega [13] applies. While we recognize a lower bound is not meant as an approximation, if we try to fit the curve $(n,M/2-Cn^{3/2})$ to the data over the interval $[10,200]$, the de la Vega curve, compared with $\left(n, M \left( \dfrac{1}{2} -\dfrac{1}{2} \sqrt{\dfrac{\log n}{\Delta_{av}}}\right)\right)$ curves upward a bit faster than the data would suggest.   In Figure 2 the left figure involves $p=.1$ and so reflects relatively low density. In the final figure (lower right), we see feedback arc sets for different values of $p$ when applied to a fixed number $n=50$ of vertices.  Here we only plot our heuristic approximation. 
\begin{figure}[ht]
    \includegraphics[width=7.5cm]{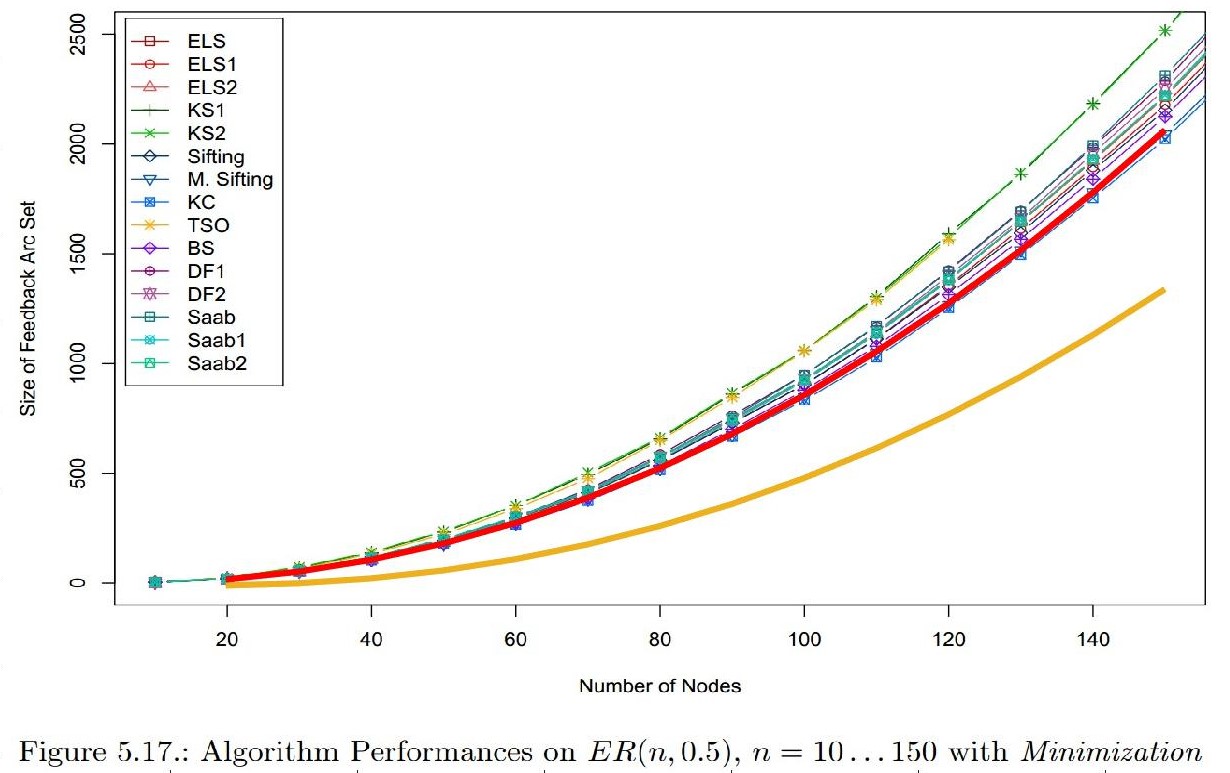}
    \includegraphics[width=8cm]{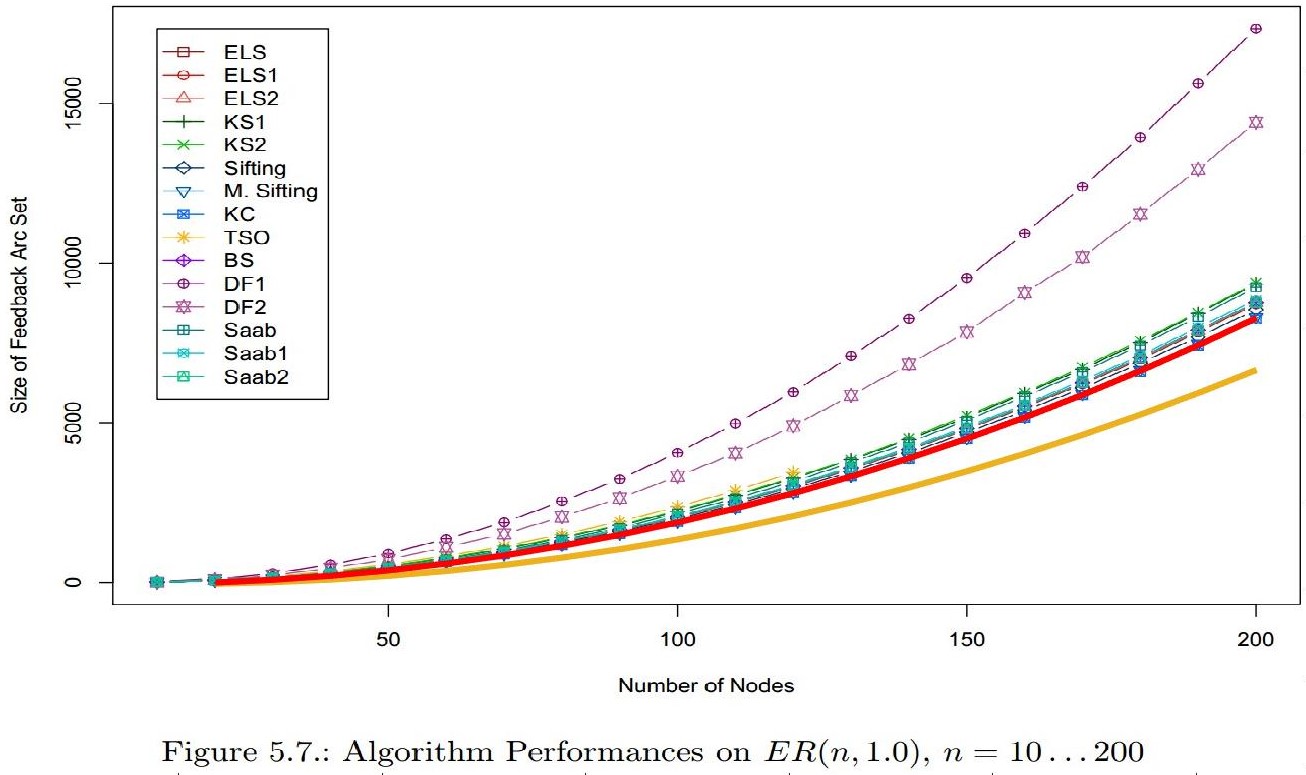}
    \captionsetup{justification=centering}
        \caption{Left:$ER(n,.5), n=10,..,150$ with \textit{Minimization} \\ Right: $ER(n,1), n=10,..,200$}
\end{figure}

\begin{figure}[ht]
    \includegraphics[width=7cm]{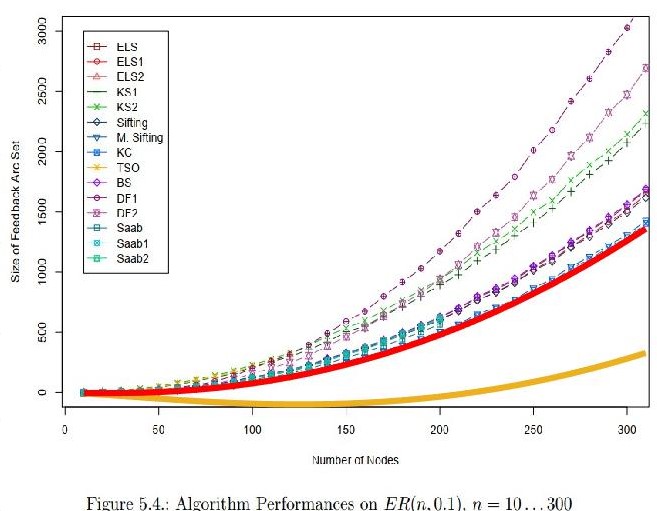}
    \includegraphics[width=8.5cm]{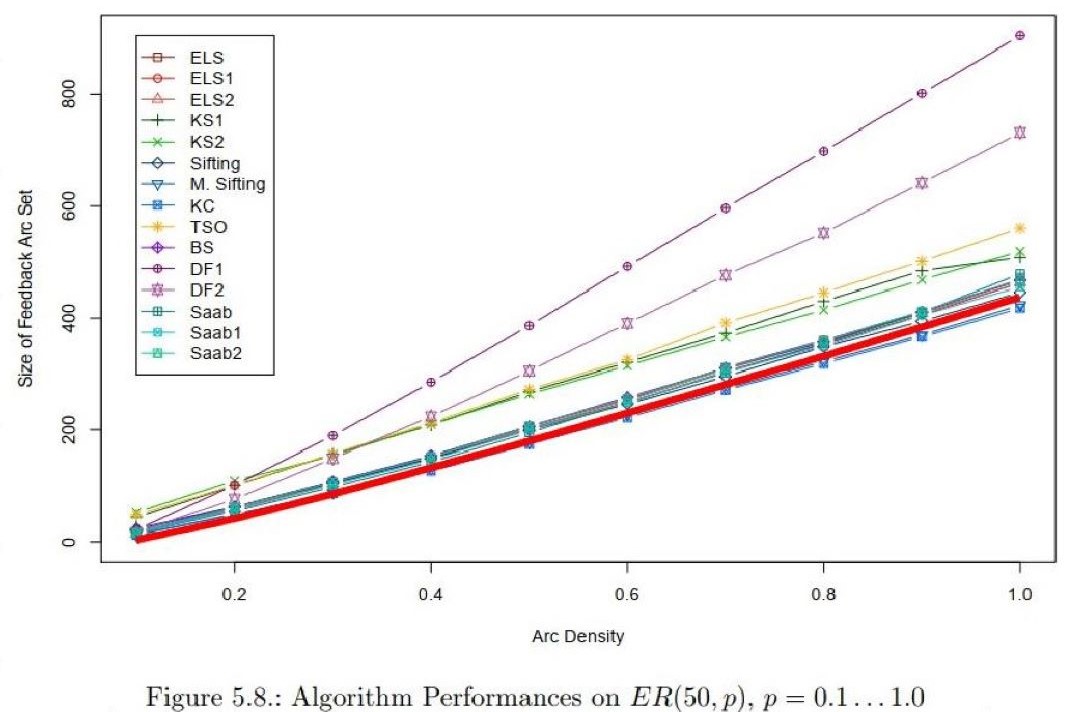}
        \captionsetup{justification=centering}
         \caption{Left: $ER(n,.1), n=20,..,300$ \\Right:$ER(50,p),p=.1,..,1.0$}
\end{figure}

It is interesting that the graphs show that the approximation works well even for reasonably small values of $n$. It would be interesting to further explore the asymptotic behavior of the expected size of the minimum feedback arc set to see whether a theoretical connection can be established with our approximate formula. We surmise as well that with more detailed arguments it may be possible to obtain a lower bound without the $\sqrt{\log n}$ factor, in light of the bounds for tournaments obtained by Spencer and de la Vega. 

\end{linenumbers}
\section {Summary}
In contemplating the minimum Feedback Arc Set problem, it is useful to know how effective reordering the vertices of an oriented directed graph would be in reducing the number of feedback arcs. We have provided here an estimate for a lower bound of the size of the minimum feedback arc set that applies with high probability, effective for small (but constant) values of $\frac{n\log n}{M}$, within a sample space for which each oriented directed graph of $M$ arcs on n vertices is equally likely. In particular, this includes large oriented directed graphs with a uniformly positive fraction of possible arcs. Plotted against experimental data, the lower bound performs as expected, and a simple heuristic approximation derived from that lower bound turns out to be surprisingly accurate in predicting an average size for minimum feedback arc sets, computed over a range of random graphs obtained using the Erd\H{o}s-R\'{e}nyi model $G(n,p)$. 

\par  
$  $

\textbf{Acknowledgement}: We wish to thank Dr. Kathrin Hanauer for providing us with a copy of her Master's Thesis [7] and for giving us permission to reproduce her figures used here as part of this paper. 


\section{References}

[1] Bonnie Berger and Peter W Shor. “Approximation alogorithms for the maximum acyclic subgraph problem”. In: \textit{Proceedings of the first annual ACM-SIAM symposium on Discrete algorithms.} 1990, pp. 236–243. 

\noindent 
[2] B Bollob{\'a}s \textit{Random Graphs.} Cambridge, UK, 2001

\noindent
[3] Harvey Diamond, Mark Kon, and Louise Raphael. \textit{Asymptotics of the Minimal Feedback Arc Set in {E}rd\H{o}s-{R}\'{e}nyi Graphs}. 2024. arXiv: 2401.04187 [math.CO].

\noindent
URL: https://arxiv.org/abs/2401.04187.

\noindent 
[4] Peter Eades, Xuemin Lin, and William F Smyth. ``A fast and effective heuristic for the feedback arc set problem'', 

\noindent In: \textit{Information processing letters} 47.6(1993), pp. 319-323.

\noindent
[5] Paul Erdos and JW Moon. ``On sets of consistent arcs in a tournament''. In: \textit{Canad. Math. Bull}   Vol 8 , Issue 3 , April 1965 , pp. 269 - 271.

\noindent
[6] Jacob Fox, Zoe Himwich, and Nitya Mani. “Extremal results on feedback arc sets in digraphs”. In: \textit{Random Structures and Algorithms} 64.2 (2024), pp. 287–308.

\noindent
[7] Kathrin Hanauer. “Algorithms for the feedback arc set problem”. Masters thesis Univ. Passau, 2010.

\noindent
[8] F. Harary. \textit{Graph Theory}. Reading, MA: Addison-Wesley, 1969.

\noindent
[9] Wassily Hoeffding. “Probability Inequalities for Sums of Bounded Random Variables”.  In: \textit{Journal of the American Statistical Association} 58.301 (1963), pp. 13–30. 

\noindent
ISSN: 01621459, 1537274X. URL: http://www.jstor.org/stable/2282952

\noindent
[10] R. Karp. “Reducibility among combinatorial problems”. In: \textit{Complexity of Computer Computations}. Ed. by R. Miller and J. Thatcher. Plenum Press, 1972, pp. 85–103.

\noindent
[11] Robert Kudeli{\'c}. \textit{Feedback arc set: a history of the problem and algorithms}. Springer Nature, 2022.

\noindent
[12] Joel Spencer. ``Optimally ranking unrankable tournaments''. In: \textit{Periodica Mathematica Hungarica} 11.2 (1980), pp. 131-144

\noindent
[13] W Fernandex de la Vega. ``On the maximum cardinality of a consistent set of arcs in a random tournament'', In: \textit{IEEE Transactions on Circuit Theory} 10.2 (1963), pp. 238-245

\noindent
[14] D Younger. “Minimum feedback arc sets for a directed graph”. In: \textit{IEEE Transactions on Circuit Theory 10.2} (1963), pp. 238–245.

\end{document}